\documentclass[letterpaper, 10 pt, conference]{ieeeconf}  

\IEEEoverridecommandlockouts                              
\usepackage{graphics} 
\usepackage{epsfig} 
\usepackage{amsmath} 
\usepackage{amssymb}  
\usepackage{xcolor}
\usepackage{subfiles}
\usepackage{todonotes}
\usepackage{tikz,pgfplots}
\usepackage{mathtools}
\usepackage{cite}
\usepackage{url}

\usetikzlibrary{shapes,positioning,calc,narrow}
\tikzset{block/.style={draw, rectangle,rounded corners, line width=1pt, minimum height=3em, minimum width=3em, outer sep=0pt}}
\tikzset{sumcircle/.style={draw, circle, outer sep=0pt, label=center:{$\sum$}, minimum width=2em}}

\DeclareMathOperator{\Tr}{Tr}
\newtheorem{theorem}{Theorem}
\newtheorem{remark}{Remark}
\usepackage{accents}
\newcommand{\ubar}[1]{\underaccent{\bar}{#1}}

\title{\LARGE \bf
On Event-Based Sampling for LQG-Optimal Control
}

\author{Marcus Thelander Andr\'{e}n, Bo Bernhardsson, Anton Cervin and Kristian Soltesz
\thanks{This work has been supported by the Swedish Research Council. 
	The authors are with the Department of Automatic Control, Lund University, Sweden, and are members of the LCCC Linnaeus Center and the ELLIIT Excellence Center at Lund University.
        {e-mail: \tt\small \{marcus, bob, anton, kristian\}@control.lth.se}}%
        }

\begin{document}

\maketitle
\thispagestyle{empty}
\pagestyle{empty}


\begin{abstract}
We consider the problem of finding an event-based sampling scheme that optimizes the trade-off between average sampling rate and control performance in a linear-quadratic-Gaussian (LQG) control problem setting with output feedback. 
Our analysis is based on a recently presented sampled-data controller structure, which remains LQG-optimal for any choice of sampling scheme.
We show that optimization of the sampling scheme is related to an elliptic convection--diffusion type partial differential equation over a domain with free boundary, a so called Stefan problem.
A numerical method is presented to solve this problem for second order systems, and thus obtain an optimal sampling scheme.
The method also directly generalizes to higher order systems, although with a higher computational cost.
For the special case of multidimensional integrator systems, we present the optimal sampling scheme on closed form, and prove that it will always outperform its periodic counterpart.
Tight bounds on the improvement are presented. 
The improved performance is also demonstrated in numerical examples, both for an integrator system and a more general case.
\end{abstract}
\medskip

\section{INTRODUCTION}
Sampled-data control has a long history of using zero-order hold (ZOH) actuation and periodic sampling, with a well-established theoretical framework \cite{astwitt97}. 
However, system resources such as network bandwidth, energy and computational power associate sampling and actuation with a cost.
The question then arises how to make more resource-efficient control implementations.
This is where the field of event-based control has emerged, where sampling and actuation are triggered only when controlled variables deviate significantly from their setpoints.
Early pioneering work in \cite{kjbob99} and \cite{arzen99} demonstrated the potential of event-based control, and it has since been a field of much activity \cite{ast08,heemels12,liu14,miskowicz15}.

In event-based control, both the intersample behavior of the control signal and the sampling scheme are considered parts of the design problem.
For a closed-loop system of the form in Fig.~\ref{fig:sys}, this corresponds to co-design of the sampler $\mathcal{S}$, the hold circuit $\mathcal{H}$, the discrete-time controller $\bar{K}$ and the sequence of sampling times $\{t_i\}_{i\in \mathbb{N}_0}$. 
Optimal co-design of these components is considered a difficult task, even for the well-studied linear-quadratic-Gaussian (LQG) control problem.
Many previous works have therefore focused on sub-optimal solutions, several of which are shown to outperform their periodic counterparts \cite{antunes12,araujo+14,antunes16,gatsis+16}.
However, these works do not consider optimization of the full co-design and typically assume setups with periodically synchronized sampling decisions and full state information.
To the authors' knowledge, there has still not been presented an optimal event-based co-design of controller and sampling scheme for the continuous-time LQG-problem with output feedback.

Recently, however, a $\mathcal{H}_2$-optimal controller structure was presented in \cite{mir17}, which was shown to be optimal for any given, uniformly bounded, sequence of sampling times.
This remarkable result was applied to event-based control in \cite{braksmayer_mir} and \cite{mir_cdc}, where in the latter it was proven that the structure is still LQG-optimal when the sampling times depend on the controlled variables.
This implies that the co-design problem is separable, and that it remains to find the sampling scheme that optimizes the performance for a given cost on sampling effort, e.g., the average sampling rate.  
In the considered setup in Fig. \ref{fig:sys}, the sampling effort represents costly data transmissions from the sensor to actuator side.

\begin{figure}[tb]
	\centering
		\begin{tikzpicture}[auto, line width=1pt, >=narrow]
	\node[block,minimum height=4em, minimum width=4em](G){$G$};
	\node[block,below=10mm of G](K){$\bar{K}$};
	\node[block,left=9mm of K](S){$\mathcal{S}$};
	\node[block,right=9mm of K](H){$\mathcal{H}$};
	\node[coordinate,right=13mm of $(G.south east)!0.25!(G.north east)$](uend){};
	\node[coordinate,left=28mm of $(G.south west)!0.75!(G.north west)$](zend){};
	\node[coordinate,right=28mm of $(G.south east)!0.75!(G.north east)$](wstart){};
	\node[coordinate,left=23mm of $(G.south west)!0.25!(G.north west)$](yturn){};
	\node[coordinate,right=23mm of $(G.south east)!0.25!(G.north east)$](uturn){};

	\draw[dotted,->](S)--node[above]{}(K);
	\draw[dotted,->](K)--node[above]{}(H);
	\draw[-]($(G.south west)!0.25!(G.north west)$)-|node[above,xshift=10mm]{$y$}(yturn);
	\draw[->](yturn)|-(S.west);
	\draw[-](H.east)-|(uturn);
	\draw[->](uturn)|-node[above,xshift=-10mm]{$u$}($(G.south east)!0.25!(G.north east)$);
	\draw[->]($(G.south west)!0.75!(G.north west)$)|-node[above,xshift=-13mm]{$z$}(zend);
	\draw[->](wstart)|-node[above,xshift=-15mm]{$w$}($(G.south east)!0.75!(G.north east)$);
	\end{tikzpicture}
	\caption{An LTI system $G$ in feedback with a sampled-data controller consisting of a sampler $\mathcal{S}$, a hold circuit $\mathcal{H}$ and a discrete-time controller $\bar{K}$. Solid lines represent continuous-time signals, whereas dashed lines represent discrete-time signals.}
	\label{fig:sys}
\end{figure}
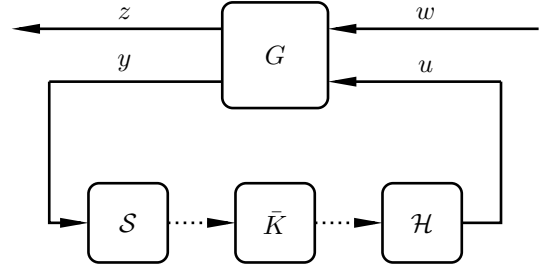

The results in \cite{mir17} also show that the closed-loop performance is fundamentally determined by a linear reset system.
Reset systems in event-based control have been considered in several works, e.g. \cite{kjbob99,meng12,toivo12}, and the results in \cite{mir17} confirm that they play an important role.
In \cite[Paper I \& II]{toivo12}, a general framework for optimal event-based impulse control is developed, and we show that this framework is key in solving the optimal sampling problem.

Our contribution is to derive the optimal event-based sampling scheme for the controller structure in \cite{mir17} using the framework in \cite{toivo12}. 
We show that the optimal sampling problem is equivalent to solving a stationary, elliptic, partial-differential equation (PDE) with free boundary, for which numerical solution methods can be applied.
For general second-order systems, we present a finite-difference method for obtaining the optimal sampling scheme numerically. 
The proposed method also directly generalizes to higher-order systems, although at a higher computational cost.
For the special case of $n$-dimensional integrator systems, we present a closed-form solution to the optimal sampling problem. 
The performance ratio between periodic and optimal event-based sampling is then shown to lie in the interval $[1\!+\!\frac{2}{n},\,3)$.


\section{PROBLEM FORMULATION}
\label{sec:problem_setup}
We consider the problem of finding the optimal sampling scheme for the closed-loop system in Fig. \ref{fig:sys}.
Here, $G$ is a linear time-invariant (LTI) continuous-time system, with an $n$-dimensional realization of the form
\begin{align}
\begin{split}
G:
\begin{cases}
\dot{x}(t)= Ax(t) + B_ww(t) + B_uu(t),\\
z(t)= C_zx(t) + D_{zu}u(t),\\
y(t)= C_yx(t) + D_{yw}w(t).
\end{cases}
\end{split}
\label{eq:sys}
\end{align}
The system $G$ is assumed to satisfy the standard assumptions for the output feedback $\mathcal{H}_2$ problem in \cite[Sec. 14.5]{zhou96} (but here we do not assume $D_{zu}$ and $D_{yw}$ to be normalized).
The inputs to $G$ are the control signal $u$ and unit-intensity white noise $w$.
The outputs are the measurement $y$ and the controlled signal $z$. 

We define the following infinite-horizon LQG cost as the measure of control performance:
\begin{align}
    J_z &\coloneqq \limsup_{T\rightarrow\infty}\frac{1}{T}\mathbb{E}[\int_0^Tz(t)^\intercal z(t)dt].
    \label{eq:z_var}
\end{align}
In the case of a linear closed-loop system, \eqref{eq:z_var} can be viewed as the squared $\mathcal{H}_2$ norm.
The average sampling rate of the system is defined as
\begin{align}
f &\coloneqq  \limsup_{T\rightarrow\infty}\frac{1}{T}\mathbb{E}[\sum_{i=0}^{\infty}1_{t_i\leq T}],
\label{eq:average_rate}
\end{align}
where the sum counts the number of sampling times up to time $T$.
The objective is to solve the optimization problem
\begin{align}
	\min_{\{t_i\}}(J_z + \rho f),
	\label{eq:opt_goal}
\end{align}
where $\rho\geq 0$ is the per-sample cost, which governs the trade-off between sampling rate and performance. 
The optimization \eqref{eq:opt_goal} is performed over the sequence of sampling times, $\{t_i\}_{i\in\mathbb{N}_0}$, based on knowledge of $y(t)$ for $t\leq t_i$ alone. 

The sampled-data controller structure $K=\mathcal{H}\bar{K}\mathcal{S}$ will be fixed to the optimal structure from \cite{mir17}, which minimizes \eqref{eq:z_var} for any sequence $\{t_i\}$.
It will be reviewed next.


\subsection{The Optimal Controller Structure}
With no cost of sampling (i.e. $\rho = 0$), the controller that minimizes $J_z$ is the continuous-time LQG controller \cite[Ch. 14]{zhou96}, with the realization 
\begin{align}
    \begin{split}
    \begin{cases}
        \dot{\hat{x}}(t) &= (A+B_uF+LC_y)\hat{x}(t) -Ly(t),\\
        u(t) &= F\hat{x}(t),
    \end{cases}    
    \end{split}
    \label{eq:lqg}
\end{align}
where $F$ and $L$ are computed by solving the two algebraic Riccati equations
\begin{align*}
    &\begin{cases}
        A^\intercal X+ XA + C_z^\intercal C_z - F^\intercal (D_{zu}^\intercal D_{zu})F = 0,\\
        F = -(D_{zu}^\intercal D_{zu})^{-1}(B_u^\intercal X + D_{zu}^\intercal C_z),
    \end{cases}
    \\
    &\begin{cases}
        AY+ YA^\intercal + B_wB_w^\intercal - L(D_{yw}D_{yw}^\intercal)L^\intercal = 0,\\
        L = -(YC_y^\intercal + B_wD_{yw}^\intercal)(D_{yw}D_{yw}^\intercal)^{-1}.
    \end{cases}
\end{align*}
The minimum cost, $\gamma_0 \coloneqq \min J_z$, is then \cite[Thm. 14.7]{zhou96}
\begin{align}
    \gamma_0 = \Tr(B_w^\intercal XB_w) + \Tr(C_zYC_z^\intercal) + \Tr(XAY + YA^\intercal X).
     \label{eq:gamma2}
\end{align}
Since sampling will always lead to some performance degradation, $\gamma_0$ is the lower bound on $J_z$ for any sampled-data controller structure $K$.
The optimal structure is given by the following theorem:

\begin{theorem}[{\cite[Thm. 5.1]{mir17}} \& {\cite[Thm. 1]{mir_cdc}}]
\label{th:opt_perf}
The optimal achievable performance for the system in Fig. \ref{fig:sys} among all causal sampled-data controllers $K$ for any uniformly bounded sampling sequence $\{t_i\}$, possibly dependent on controlled variables, is
\begin{equation}
     \min_K J_z = \gamma_0 + \limsup_{T\rightarrow\infty}\frac{1}{T}\mathbb{E}[\int_0^T\eta(t)^\intercal \eta(t)dt],
     \label{eq:optimal_cost}
\end{equation}
where $\eta$ is the output of a linear reset system $H$
\begin{align}
    \begin{split}    
    H:\begin{cases}
    \dot{x}_{\textsc{\tiny H}}(t) &= Ax_{\textsc{\tiny H}}(t) + \epsilon(t),\quad\quad\quad  x_{\textsc{\tiny H}}(t_i) = 0,\\
			\eta(t) &= (D_{zu}^\intercal D_{zu})^{\frac{1}{2}}Fx_{\textsc{\tiny H}}(t).    
\end{cases}
\end{split}
\label{eq:h_sys} 
\end{align}   
The input $\epsilon = L(y-C_yx_s)$ is the measurement update of the Kalman--Bucy filter 
\begin{align}
    \dot{x}_s(t) &= Ax_s(t) + B_uu(t) - L(y(t) - C_yx_s(t)),
    \label{eq:kalman}
\end{align}
operating on the sensor-side measurement $y$. Furthermore, $x_{\textsc{\tiny H}} = x_a-x_s$, where $x_a$ is the state vector of an intermittently reset LQR controller simulating the closed loop on the actuator side,
\begin{align}
\begin{split}
\begin{cases}
    \dot{x}_a(t) &= (A+B_uF)x_a(t),\quad\quad x_a(t_i) = x_s(t_i), \\
    u(t) &= Fx_a(t).
\end{cases}
\end{split}
    \label{eq:lqr}
\end{align}
\end{theorem}
\smallskip
\begin{proof}
See proof of Theorem 5.1 in \cite{mir17} for the original derivation, and proof of Theorem 3.1 in \cite{mir_cdc} for the extension to stochastic disturbances and sampling sequences that are allowed to depend on controlled variables. 
The scaling factor $(D_{zu}^\intercal D_{zu})^{\frac{1}{2}}$ in \eqref{eq:h_sys} is featured here since we do not assume the normalization $D_{zu}^\intercal D_{zu} = I$ as in \cite{mir17}.
\end{proof}

\begin{remark}
	The optimal controller structure also has a representation given explicitly in $\mathcal{S}$, $\bar{K}$ and $\mathcal{H}$, derived in \cite[Corollary 4.3]{mir17}. 
	It would be considered for practical implementations, but is less intuitive for analysis.
\end{remark}

A block diagram of the optimal controller structure \eqref{eq:kalman}--\eqref{eq:lqr} is presented in Fig. \ref{fig:converted_sys}.
Naturally, a copy of \eqref{eq:lqr} should be featured at the sensor side to compute the control signal $u$ for the Kalman--Bucy filter.

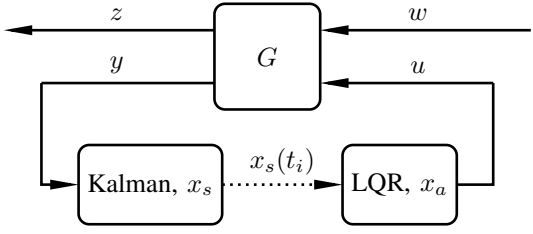
\begin{figure}[tb]
	\centering
    	\begin{tikzpicture}[auto, line width=1pt, >=narrow]
	\node[block,minimum height=4em, minimum width=4em](G){$G$};
	\node[coordinate,below=10mm of G](K){};
	\node[block,left=6mm of K](S){Kalman, $x_s$};
	\node[block,right=10mm of K](H){LQR, $x_a$};
	\node[coordinate,right=13mm of $(G.south east)!0.25!(G.north east)$](uend){};
	\node[coordinate,left=28mm of $(G.south west)!0.75!(G.north west)$](zend){};
	\node[coordinate,right=28mm of $(G.south east)!0.75!(G.north east)$](wstart){};
	\node[coordinate,left=23mm of $(G.south west)!0.25!(G.north west)$](yturn){};
	\node[coordinate,right=23mm of $(G.south east)!0.25!(G.north east)$](uturn){};

	\draw[dotted,->](S)--node[above]{$x_s(t_i)$}(H);
	\draw[-]($(G.south west)!0.25!(G.north west)$)-|node[above,xshift=10mm]{$y$}(yturn);
	\draw[->](yturn)|-(S.west);
	\draw[-](H.east)-|(uturn);
	\draw[->](uturn)|-node[above,xshift=-10mm]{$u$}($(G.south east)!0.25!(G.north east)$);
	\draw[->]($(G.south west)!0.75!(G.north west)$)|-node[above,xshift=-13mm]{$z$}(zend);
	\draw[->](wstart)|-node[above,xshift=-15mm]{$w$}($(G.south east)!0.75!(G.north east)$);
	\end{tikzpicture}
    \caption{The optimal sampled-data controller structure. 
    	It features a Kalman--Bucy filter on the sensor side that intermittently transmits its estimate to an LQR controller simulating the closed-loop system on the actuator side.}
    \label{fig:converted_sys}
\end{figure}

The system $H$ from Theorem \ref{th:opt_perf} describes the inter-sample degradation compared to the performance of the continuous-time LQG controller.
It is driven by the signal $\epsilon$, which will be a white process with intensity $LD_{yw}D_{yw}^\intercal L^\intercal$.

\subsection{The Optimal Sampling Problem}
Using \eqref{eq:optimal_cost} in Theorem \ref{th:opt_perf}, we can rewrite the original problem \eqref{eq:opt_goal} as
\begin{align*}    
    \min_{\{t_i\}}&(J_z + \rho f) = \gamma_0 + \min_{\{t_i\}}(J_{\textsc{\tiny H}} + \rho f),\\
    J_{\textsc{\tiny H}} &\coloneqq \limsup_{T\rightarrow\infty}\frac{1}{T}\mathbb{E}[\int_0^T\eta(t)^\intercal \eta(t)dt],
\end{align*}
i.e., it is sufficient to consider optimization over the linear reset system in \eqref{eq:h_sys}.
Therefore we construct a sampling scheme based on the signal $x_{\textsc{\tiny H}}$, $\phi(x_{\textsc{\tiny H}}) : \mathbb{R}^n \rightarrow \{0,1\}$, where $1$ denotes ``sample'' and $0$ denotes ``no sample'', i.e., $\{t_i\} = \{t\,|\, \phi(x_{\textsc{\tiny H}}(t))=1\}$.
The problem of finding the optimal sampling scheme is thus 
\begin{align}
    \begin{split}    
    &\min_{\phi}(J_{\textsc{\tiny H}} + \rho f).
    \label{eq:opt_problem}
    \end{split}
\end{align}
\section{A FRAMEWORK FOR THE OPTIMAL SAMPLING PROBLEM}
\label{sec:optimal_bounds}
A problem similar to \eqref{eq:opt_problem} was considered in \cite[Paper I and II]{toivo12}, and we will build upon the framework presented there.
It considers optimal event-based impulse control of LTI systems, driven by a white process with intensity $R$, where the system state $x$ is reset to zero at each actuation. 
The cost considered in \cite{toivo12} has the form
\begin{align*}
	\limsup_{T\rightarrow\infty} \frac{1}{T}\mathbb{E}[\int_{0}^{T}x^\intercal(t)Qx(t)dt] + \rho f,
\end{align*}
with $f$ and $\rho$ defined as in \eqref{eq:average_rate}, \eqref{eq:opt_goal}. 
The state cost matrix $Q$ is considered a design variable. 
Note that if we consider the choice $Q = F^\intercal D_{zu}^\intercal D_{zu}F$, this cost is equivalent to the one in \eqref{eq:opt_problem}.
Furthermore, with $R = LD_{yw}D_{yw}^\intercal L^\intercal$ the system configuration is identical to \eqref{eq:h_sys}.
The framework of \cite[Paper I and II]{toivo12} is thus directly applicable to \eqref{eq:opt_problem}.

\subsection{The Equivalent Value Function Problem}
The framework in \cite[Paper I and II]{toivo12} is based on introducing and optimizing bounds on the cost, $\ubar{J} \leq J_{\textsc{\tiny H}} + \rho f \leq \bar{J}$. 
The bounds are determined by finding a bounded, $\mathcal{C}^2$, value function $V : x_{\textsc{\tiny H}}\rightarrow\mathbb{R}$, satisfying the path constraints
\begin{align*}
\begin{cases}
	\frac{1}{T}(j + \mathbb{E}[V(x_{\textsc{\tiny H}}(T)) - V(x_{\textsc{\tiny H}}(0))]) \geq \ubar{J},\\
	\frac{1}{T}(j + \mathbb{E}[V(x_{\textsc{\tiny H}}(T)) - V(x_{\textsc{\tiny H}}(0))]) \leq  \bar{J}, \\
\end{cases}\\
	j \eqqcolon \mathbb{E}[\int_{0}^{T}x_{\textsc{\tiny H}}^\intercal(t)Qx_{\textsc{\tiny H}}(t)dt] + \rho\mathbb{E}[\sum_{i=0}^{\infty}1_{t_i\leq T}].
\end{align*}
If such a function $V$ is found for a given sampling scheme, then the limit $T\rightarrow\infty$ gives the desired bounds on $J_{\textsc{\tiny H}} + \rho f$.
The objective now is twofold; to minimize $\bar{J}$ with respect to the sampling scheme and to find a value function $V$ such that the bounds are tight, i.e., $\ubar{J} = \bar{J}$.
Achieving both means that we have solved \eqref{eq:opt_problem}.
The following theorem specifies when this is the case.

\begin{theorem}[{\cite[Paper II, Thm. 1]{toivo12}}]
Consider the optimal sampling problem in \eqref{eq:opt_problem}.
Suppose a bounded, $\mathcal{C}^2$,  function $V$ and a constant $J$ satisfy
\begin{align}
    x_{\textsc{\tiny H}}^\intercal Qx_{\textsc{\tiny H}} + x_{\textsc{\tiny H}}^\intercal A^\intercal \nabla V + \frac{1}{2}\Tr(R\nabla^2V) \geq J,\,\,  \forall x_{\textsc{\tiny H}}\in \mathbb{R}^n,
    \label{eq:v_1}
    \\
    \rho \geq V(x_{\textsc{\tiny H}})-V(0), \quad\,\,\,\, \forall x_{\textsc{\tiny H}}\in \mathbb{R}^n,
    \label{eq:v_2}
\end{align}
where equality is achieved in at least one of \eqref{eq:v_1} and \eqref{eq:v_2} for each $x_{\textsc{\tiny H}}$. 
Then the optimal cost in \eqref{eq:opt_problem} is $J$, and
\begin{align}
	\frac{1}{T}(j + \mathbb{E}[V(x_{\textsc{\tiny H}}(T)) - V(x_{\textsc{\tiny H}}(0))]) \geq J.
	\label{eq:path_constraint}
\end{align}
Equality is achieved in \eqref{eq:path_constraint} when sampling is triggered on equality in \eqref{eq:v_2}, i.e., an optimal sampling scheme is
\begin{align*}
    \phi(x_{\textsc{\tiny H}}) = \begin{cases}
        1, \quad V(x_{\textsc{\tiny H}})- V(0) = \rho,\\
        0, \quad \text{ otherwise.}
    \end{cases}
\end{align*}
\label{th:bob_toivo}
\end{theorem}
\begin{proof}
See proof of Theorem 1 in \cite[Paper II]{toivo12}.
\end{proof}


The inequalities \eqref{eq:v_1}, \eqref{eq:v_2} can also be formulated as a stationary convection--diffusion type PDE over a domain with a free boundary, a so called Stefan problem \cite{stefan}. 
Such problems are usually found in models describing mediums undergoing phase change, e.g., melting ice.

Let $\Omega$ denote the set where equality is attained in \eqref{eq:v_1}. 
On the boundary, $\partial \Omega$, we have equality in \eqref{eq:v_2}, meaning that $\partial \Omega$ is a level set of $V$. 
Since $V$ is $\mathcal{C}^2$ and constant outside $\Omega$ (equality in \eqref{eq:v_2}), it holds that  $\nabla V = 0$ on $\partial \Omega$. 
We can thus write \eqref{eq:v_1}, \eqref{eq:v_2} as 
\begin{align}
	\begin{split}
	\begin{cases}
	x_{\textsc{\tiny H}}^\intercal Qx_{\textsc{\tiny H}} - J + x_{\textsc{\tiny H}}^\intercal A^\intercal \nabla V + \frac{1}{2}\Tr(R\nabla^2V) = 0,   
	\\
	 V(x_{\textsc{\tiny H}})\leq\rho + V(0),\,\,\quad\quad\quad\quad\quad\,\,\,\,\, \forall x_{\textsc{\tiny H}}\in \mathbb{R}^n,
	\end{cases}
		\end{split}
	    \label{eq:v_1s} 
	    \\
	    \begin{split}
	    \begin{cases}
	 V(x_{\textsc{\tiny H}}) = \rho + V(0), \\
	 \nabla V = 0,\quad\quad\quad\quad\quad\quad\quad\quad\quad\,\,\,\,\,\,\,\, \forall x_{\textsc{\tiny H}} \in \partial \Omega.
	 \end{cases}
    \end{split}
     \label{eq:v_2s}
\end{align}
For fixed values of $J$, $\rho$, and $V(0)$, this is a Stefan problem in $V$ with the free boundary $\partial \Omega$. 
The first two terms in \eqref{eq:v_1s} describe production, the third convection and the fourth diffusion. 
On the free boundary $\partial \Omega$ we have both a Dirichlet and a Neumann condition, given by \eqref{eq:v_2s}. 
While PDEs of this kind generally do not have a closed-form solution, there exist several numerical solution methods (see e.g., \cite{bernauer12,barrett14}).

\subsection{A Closed-Form Solution for the Integrator Case}
In the special case of an $n$-dimensional integrator system, there is in fact a closed-form solution to \eqref{eq:v_1}, \eqref{eq:v_2}.
\begin{theorem}[{\cite[Paper II, Thm. 2]{toivo12}}]
For $A=0$, the value function $V$ satisfying \eqref{eq:v_1} and \eqref{eq:v_2} is given by
\begin{align*}
    V(x_{\textsc{\tiny H}}) &= \begin{cases}-\frac{1}{4}g(x_{\textsc{\tiny H}})^2, &g(x_{\textsc{\tiny H}}) \geq 0,\\
    0, &\text{ otherwise,}
    \end{cases}
    \\
    g(x_{\textsc{\tiny H}}) &= 2\sqrt{\rho} - x_{\textsc{\tiny H}}^\intercal Px_{\textsc{\tiny H}},
\end{align*}
which corresponds to the optimal sampling scheme
\begin{align*}
    \phi(x_{\textsc{\tiny H}}) = \begin{cases}1, \quad 2\sqrt{\rho} - x_{\textsc{\tiny H}}^\intercal Px_{\textsc{\tiny H}} = 0,\\
    0, \quad\quad\quad\quad\,\,\, \text{ otherwise.}
    \end{cases} 
\end{align*}
For $R\succ0$, $Q\succ0$, the matrix $P\succ0$ is the unique solution of the Riccati-like equation
\begin{align}
    PRP + \frac{1}{2}\Tr(RP)P = Q.
    \label{eq:riccatilike}
\end{align}
The optimal cost is $J = \sqrt{\rho}\Tr(RP)$, where specifically
\begin{align}
    J_{\textsc{\tiny H}} &= \rho f =  \frac{1}{2}\sqrt{\rho}\Tr(RP).
    \label{eq:specific_cost}
\end{align}
\label{th:integrator}
\end{theorem}
\begin{proof}
    All statements but \eqref{eq:specific_cost} and the uniqueness of $P$ are proven by insertion into \eqref{eq:v_1} and \eqref{eq:v_2} in Theorem \ref{th:bob_toivo}. 
    For a detailed proof, see \cite[Paper II]{toivo12}.
\end{proof}
\begin{remark}
	Positive definiteness of $Q$ and $R$ in \eqref{eq:riccatilike} is assured by the standard assumptions on \eqref{eq:sys}, since $A=0$ implies $R = LD_{yw}D_{yw}^\intercal L^\intercal = B_wB_w^\intercal \succ 0$ and $Q = F^\intercal D_{zu}^\intercal D_{zu}F = C_z^\intercal C_z \succ0$.
\end{remark}

The shapes of the ellipsoidal trigger bounds, illustrated in Fig. \ref{fig:a0_shapes}, reflect the optimal trade-off between sampling rate and performance.
Directions associated with large eigenvalues of the noise intensity matrix $R$ will typically have a bound moved further away from the origin, sacrificing some performance to ensure fewer sampling events.
The opposite is true for directions associated with large eigenvalues of the state cost matrix $Q$.

\begin{figure}[tb]
	\centering
    \input{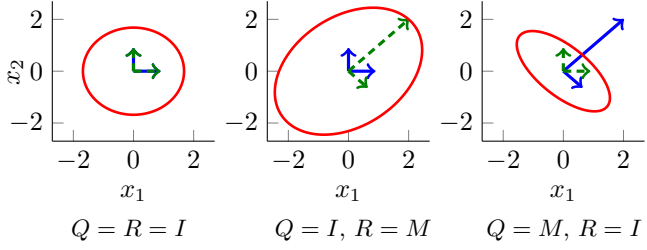}
    \put(-217,-10){\small$Q=R=I$}
    \put(-140,-10){\small$Q=I,\,R=M$}
    \put(-60,-10){\small$Q=M,\,R=I$}
    \caption{Examples of the optimal, ellipsoidal, trigger bounds (red, solid) in the integrator case. 
    The state cost matrix $Q$ (blue, solid) and noise intensity matrix $R$ (green, dashed) are represented by their respective eigenvectors scaled according to associated eigenvalues.
	For all cases $\rho = 1$ and $M=N_{\pi/4}^\intercal\text{diag}(1,10)N_{\pi/4}$, where $N_{\theta}$ is the rotation matrix defined in \eqref{eq:rotation_matrix}.}
    \label{fig:a0_shapes}
\end{figure}

Since the Riccati-like equation \eqref{eq:riccatilike} is nonlinear in $P$, it cannot be solved using semi-definite programming.
However, the following result reduces the problem to a scalar search.
\begin{theorem}
	\label{th:scalarform}
	Consider equation \eqref{eq:riccatilike}.
	Define the matrix $\tilde{R} \coloneqq Q^{\frac{1}{2}}RQ^{\frac{\intercal}{2}}\succ0$, and introduce the decomposition
	\begin{align*}
		\tilde{R} = U\bar{R}U^\intercal, \quad UU^\intercal=I, \quad \bar{R} = \text{diag}\{r_i\}. 
	\end{align*}
	The solution to \eqref{eq:riccatilike} is then given by 
	\begin{align*}
		P = Q^\frac{\intercal}{2}U\bar{P}U^\intercal Q^\frac{1}{2}, \quad \bar{P} = \text{diag}\{p_i\},
	\end{align*}
	where
	\begin{align*}
		p_i = -\frac{s}{4r_i} + \sqrt{\frac{s^2}{16r_i^2} + \frac{1}{r_i}},\quad i\in\{1,\,...,\,n\},
	\end{align*}
	and $s>0$ is the unique root of the scalar function
	\begin{align*}
		h(s) \coloneqq (n+4)s - \sum_{j=1}^{n}\sqrt{s^2 + 16r_j}.
	\end{align*}
\end{theorem}
\begin{proof}
The proof is presented in the appendix.
\end{proof}
  
For the integrator case, we can use Theorem \ref{th:integrator} to compute the cost as a function of the average sampling period $h_{\text{avg}} \coloneqq 1/f$
\begin{align}
  J_{\textsc{\tiny H}} = J_{\text{e}}h_{\text{avg}}, \quad J_{e} = \frac{[\Tr(RP)]^2}{4}.
  \label{eq:optperf_vs_samp}
\end{align}
With this result, it is interesting to compare how much optimal event-based sampling improves the performance over periodic sampling. 
The cost using periodic sampling is given by \cite[Remark 4]{braksmayer_mir}
\begin{align*}
      J_{\textsc{\tiny H}} = \frac{1}{h}\int_0^h\int_0^{h-\tau}\|D_{zu}Fe^{A\sigma}LD_{yw}\|_F^2d\sigma d\tau,
\end{align*}
where $h=h_{\text{avg}}$ is the sampling period. 
With $A=0$ we can rewrite this as
\begin{align}
     J_{\textsc{\tiny H}} = J_{\text{p}}h_{\text{avg}},\quad J_{\text{p}}= \frac{\Tr(RQ)}{2},
     \label{eq:perpef_vs_samp}
\end{align}
by noting that $\|D_{zu}FLD_{yw}\|_F^2 = \Tr(RQ)$. 
Based on \eqref{eq:optperf_vs_samp} and \eqref{eq:perpef_vs_samp}, we can prove that optimal event-based sampling will always outperform its periodic counterpart for the integrator case.
\begin{theorem}
	 \label{th:bounds}
    Using the optimal sampled-data controller structure \eqref{eq:kalman}, \eqref{eq:lqr}, the performance ratio, $J_{\text{ratio}} \coloneqq J_{\text{p}}/J_{\text{e}}$, between periodic and optimal event-based sampling for the $n$-dimensional integrator system is bounded by
        \begin{align}
        1 + \frac{2}{n} \leq J_{\text{ratio}} &< 3, \quad n\geq 2,
        \label{eq:bounds}
    \end{align}
	and $J_{\text{ratio}}=3$ for the special case $n=1$. 
    Furthermore, the bounds in \eqref{eq:bounds} are tight, with 
    equality in the lower bound when the eigenvalues of the matrix $RQ$ are identical.
    The ratio $J_{\text{ratio}}$ approaches the upper bound when all but one eigenvalue of $RQ$ approach zero.
\end{theorem}
\begin{proof}
The proof is presented in the appendix.	
\end{proof}

\begin{figure*}[tb]
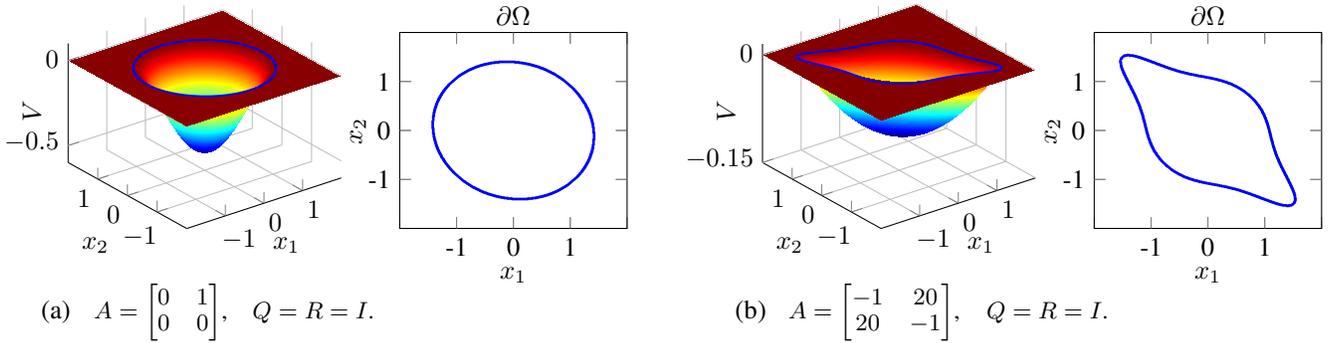

	\centering
	\input{convex3.tex}
	\put(-225,-10){(a)}
	\put(-205,-10){\small$A=\begin{bmatrix}
		0 & 1\\ 0 &0
		\end{bmatrix}$,}
	\put(-145,-10){\small$Q=R=I$.}\hspace{12pt}
	\input{nonconvex3.tex}
	\put(-225,-10){(b)}
	\put(-205,-10){\small$A=\begin{bmatrix}
		-1 & 20\\ 20 &-1
		\end{bmatrix}$,}
	\put(-130,-10){\small$Q=R=I$.}
	\caption{Numerically obtained value function $V$ (surface plot) and corresponding trigger bound $\partial \Omega$ (blue, solid) for (a) a double integrator system  and (b) a system with strong cross-coupling of the states. Note that the trigger bound is convex in (a) but non-convex in (b).}
	\label{fig:bounds}
\end{figure*}

\subsection{A Numerical Method for the General Case}
\label{sec:numerical_method}
A closed-form solution for the general case, $A\neq0$, remains to be found, if it even exists. However, we have been able to obtain $V$ numerically for the second-order case. The numerical solution gives some further intuition into the optimal event-based control strategy.

Consider the the Stefan problem of the form \eqref{eq:v_1n}--\eqref{eq:v_2n}. 
Assume that we wish to solve the non-stationary version of the problem, for some fixed $J$ and with $V(0,t) := -\rho$, $\forall t$ (the choice of $V(0,t)$ is not  essential, as it only acts as a reference point). The problem then has the form
\begin{align}
\begin{split}
\begin{cases}
x_{\textsc{\tiny H}}^\intercal Qx_{\textsc{\tiny H}} - J + x_{\textsc{\tiny H}}^\intercal A^\intercal \nabla V + \frac{1}{2} \Tr(R\nabla^2V) = \frac{\partial V}{\partial t},   
\\
V(x_{\textsc{\tiny H}},t) \leq 0,\quad\quad\quad\quad\quad\quad\quad\,\,\quad \forall x_{\textsc{\tiny H}}\in \mathbb{R}^n, \forall t,
\end{cases}
\end{split}
\label{eq:v_1n} 
\\
\begin{split}
\begin{cases}
V(x_{\textsc{\tiny H}},t) = 0,\\
\nabla V = 0,\quad\quad\quad\quad\quad\quad\quad\quad\quad\quad\, \forall x_{\textsc{\tiny H}} \in \partial \Omega, \forall t.
\end{cases}
\end{split}
\label{eq:v_2n}
\end{align}
A natural idea is to use standard numerical tools for PDE solving to simulate \eqref{eq:v_1n} in discrete time, while at each time step enforcing the inequality on $V$.
To approximate the differential operators in \eqref{eq:v_1n}, we use the backward-time central-space (BTCS) finite-difference method \cite{recktenwald11} and enforce the inequality by assigning $V \coloneqq \min(V,0)$ at each time step.
This assignment has the effect of moving the free boundary $\partial\Omega$ over time. 
The simulation is then progressed until a stopping criterion based on stationarity of $V$ is met. 
If stationarity is reached, this means that we have found a solution satisfying \eqref{eq:v_1s}, \eqref{eq:v_2s} for the specific $J$, and thus the optimal trigger bound is given by the resulting boundary $\partial\Omega$.
For further detail, we refer to the supplemental Matlab code\footnote{Available at {https://gitlab.control.lth.se/marcus/optimal-trigger-bound.git}}.

The method has been validated using the known results for the case $A=0$ and is successfully reaching a stationary solution for all second-order systems considered so far.
While it could in principle be applied to higher-order cases as well, this becomes unwieldy due to the time complexity, which grows exponentially in $n$. 
For second-order systems, preliminary findings suggest that the trigger bound is convex and almost ellipsoidal whenever the system states are relatively weakly coupled.
An example of this situation is seen in Fig. \ref{fig:bounds}a, which shows the resulting $V$ and trigger bound for a double (chain) integrator.
However, the optimal trigger bound is not convex in general, as seen in Fig. \ref{fig:bounds}b, where $V$ and the optimal trigger bound are presented for a system with a strong cross-coupling between the two states.

\section{NUMERICAL EXAMPLES}
Here we demonstrate the described optimal co-design of controller structure and event-based sampling through two numerical second-order examples---one integrator and one unstable system. 
We compare the trade-off between performance $J_z$ and average sampling period $h_{\text{avg}}$ for three different controller structures and sampling methods:
\begin{enumerate}[\setlabelwidth{12)}]
	\item[(a)] Standard zero-order hold LQG using periodic sampling.
	\item[(b)] The optimal controller structure, \eqref{eq:kalman}, \eqref{eq:lqr}, using periodic sampling.
	\item[(c)] The optimal controller structure, \eqref{eq:kalman}, \eqref{eq:lqr}, using optimal event-based sampling.
\end{enumerate} 

\subsection{Integrator Example}
For the integrator example, $A=0$, we consider a second-order system with the following parameters:
\begin{align*}
	B_w &= \begin{bmatrix}
	(R_d^{\frac{1}{2}}N_{\pi/8})^\intercal & 0
	\end{bmatrix}, &B_u &= C_y= I, \\
	C_z &= \begin{bmatrix}
	(Q_d^{\frac{1}{2}}N_{\pi/4})^\intercal &
	0
	\end{bmatrix}^\intercal, &D_{yw}&=D_{zu}^\intercal = \begin{bmatrix}
	0  &
	I
	\end{bmatrix}, 
\end{align*}
where
\begin{align}
	R_d &= Q_d = \begin{bmatrix}
	1 & 0\\
	0 & 5
	\end{bmatrix}, & N_{\theta} &= \begin{bmatrix}
	\cos(\theta) & -\sin(\theta) \\
	\sin(\theta) & \cos(\theta)
	\end{bmatrix}.
	\label{eq:rotation_matrix}
\end{align} 
This choice of $C_z$ and $B_w$ corresponds to
\begin{align*}
	Q &= N_{\pi/4}^\intercal Q_dN_{\pi/4}, & R &=N_{\pi/8}^\intercal R_dN_{\pi/8}, 
\end{align*}
i.e., the corresponding linear reset system $H$ will have asymmetric cost and noise intensity matrices whose eigenvectors are not aligned to each other nor the state axes.
This ensures that the integrator states are coupled through the input noise and controlled output, as to avoid the degenerated case of two uncoupled first-order integrator processes.

Trade-off curves for (b) and (c) are obtained through \eqref{eq:optperf_vs_samp} and \eqref{eq:perpef_vs_samp} respectively. 
For (a) we use the Matlab toolbox Jitterbug \cite{jitterbug02}. 

The resulting trade-off curves are presented in Fig. \ref{fig:a0example}. 
Note that the improvement made by going from the sub-optimal structure (a) to the optimal structure (b) is smaller than the improvement by going from periodic sampling (b) to optimal event-based sampling (c). 
This suggests that sampling scheme, rather than controller structure, is the main performance contributor in this example. 

\begin{figure}[tb]
	\centering
%
%
\begin{tikzpicture}

\begin{axis}[%
width=0.78\columnwidth,
scale only axis,
xmin=0,
xmax=1.4,
xlabel={Average Sampling Period, $h_{\text{avg}}$},
x label style={at={(axis description cs:0.5,0.03)},anchor=north},
ytick={1,1.25,1.5,1.75,2,2.25},
yticklabels = {$\gamma_0$,1.25$\gamma_0$,1.5$\gamma_0$,1.75$\gamma_0$,2$\gamma_0$,2.25$\gamma_0$},
xmajorgrids,
ymin=1,
ymax=2.2,
ylabel={Cost, $J_z$},
ylabel style={rotate=0,at={(axis description cs:-0.05,0.5)},anchor=north},
ymajorgrids,
legend style={at={(0.0,1)},anchor=north west,draw=black,fill=white,legend cell align=left,font=\small}
]
\addplot [color=blue,solid,line width=1.0pt]
  table[row sep=crcr]{%
0.001	1.00051633522843\\
0.025	1.01296395600261\\
0.05	1.02604365007438\\
0.075	1.03923895016053\\
0.1	1.05254963673833\\
0.125	1.06597540366431\\
0.15	1.07951585934985\\
0.175	1.09317052825439\\
0.2	1.10693885268306\\
0.225	1.12082019487314\\
0.25	1.13481383935118\\
0.275	1.14891899554092\\
0.3	1.16313480060039\\
0.325	1.17746032246499\\
0.35	1.19189456307265\\
0.375	1.20643646174622\\
0.4	1.22108489870781\\
0.425	1.23583869869994\\
0.45	1.25069663468832\\
0.475	1.26565743162169\\
0.5	1.28071977022493\\
0.525	1.29588229080251\\
0.55	1.31114359703061\\
0.575	1.32650225971747\\
0.6	1.34195682051314\\
0.625	1.35750579555119\\
0.65	1.37314767900671\\
0.675	1.38888094655657\\
0.7	1.40470405872966\\
0.725	1.42061546413633\\
0.75	1.43661360256826\\
0.775	1.45269690796131\\
0.8	1.46886381121547\\
0.825	1.4851127428678\\
0.85	1.50144213561516\\
0.875	1.51785042668527\\
0.9	1.53433606005554\\
0.925	1.55089748852022\\
0.95	1.56753317560751\\
0.975	1.58424159734899\\
1	1.60102124390458\\
1.025	1.61787062104674\\
1.05	1.63478825150846\\
1.075	1.65177267619958\\
1.1	1.6688224552969\\
1.125	1.68593616921336\\
1.15	1.70311241945203\\
1.175	1.72034982935073\\
1.2	1.73764704472319\\
1.225	1.75500273440266\\
1.25	1.77241559069382\\
1.275	1.78988432973902\\
1.3	1.80740769180432\\
1.325	1.82498444149118\\
1.35	1.84261336787915\\
1.375	1.86029328460476\\
1.4	1.87802302988191\\
1.425	1.89580146646835\\
1.45	1.9136274815832\\
1.475	1.93149998677967\\
1.5	1.9494179177774\\
1.525	1.96738023425823\\
1.55	1.98538591962934\\
1.575	2.00343398075705\\
1.6	2.02152344767479\\
1.625	2.03965337326819\\
1.65	2.05782283294028\\
1.675	2.07603092425938\\
1.7	2.09427676659218\\
1.725	2.11255950072439\\
1.75	2.13087828847094\\
1.775	2.1492323122778\\
1.8	2.16762077481713\\
1.825	2.18604289857745\\
1.85	2.2044979254502\\
1.875	2.22298511631434\\
1.9	2.24150375061976\\
1.925	2.2600531259711\\
1.95	2.27863255771263\\
1.975	2.29724137851527\\
2	2.31587893796646\\
};
\addlegendentry{(a) ZOH LQG, Periodic};

\addplot [color=red,densely dashdotted,line width=1pt]
  table[row sep=crcr]{%
0	1\\
0.1	1.0516242597217\\
0.2	1.10324851944339\\
0.3	1.15487277916509\\
0.4	1.20649703888679\\
0.5	1.25812129860849\\
0.6	1.30974555833018\\
0.7	1.36136981805188\\
0.8	1.41299407777358\\
0.9	1.46461833749528\\
1	1.51624259721697\\
1.1	1.56786685693867\\
1.2	1.61949111666037\\
1.3	1.67111537638207\\
1.4	1.72273963610376\\
1.5	1.77436389582546\\
1.6	1.82598815554716\\
1.7	1.87761241526886\\
1.8	1.92923667499055\\
1.9	1.98086093471225\\
2	2.03248519443395\\
2.1	2.08410945415565\\
2.2	2.13573371387734\\
2.3	2.18735797359904\\
2.4	2.23898223332074\\
2.5	2.29060649304243\\
2.6	2.34223075276413\\
2.7	2.39385501248583\\
2.8	2.44547927220753\\
2.9	2.49710353192922\\
3	2.54872779165092\\
};
\addlegendentry{$\text{(b) Opt. Struct., Periodic, }\gamma{}_\text{0}\text{ + 11.83h}_{\text{avg}}$};

\addplot [color=black,densely dashed,line width=1.0pt]
  table[row sep=crcr]{%
0	1\\
0.1	1.01961152618974\\
0.2	1.03922305237948\\
0.3	1.05883457856922\\
0.4	1.07844610475896\\
0.5	1.0980576309487\\
0.6	1.11766915713844\\
0.7	1.13728068332818\\
0.8	1.15689220951792\\
0.9	1.17650373570766\\
1	1.1961152618974\\
1.1	1.21572678808713\\
1.2	1.23533831427687\\ 
1.3	1.25494984046661\\
1.4	1.27456136665635\\
1.5	1.29417289284609\\
1.6	1.31378441903583\\
1.7	1.33339594522557\\
1.8	1.35300747141531\\
1.9	1.37261899760505\\
2	1.39223052379479\\
2.1	1.41184204998453\\
2.2	1.43145357617427\\
2.3	1.45106510236401\\
2.4	1.47067662855375\\
2.5	1.49028815474349\\
2.6	1.50989968093323\\
2.7	1.52951120712297\\
2.8	1.54912273331271\\
2.9	1.56873425950245\\
3	1.58834578569219\\
};
\addlegendentry{$\text{(c) Opt. Struct., Event-Based, }\gamma{}_\text{0}\text{ + 4.49h}_{\text{avg}}$};

\end{axis}
\end{tikzpicture}%
	\caption{Trade-off curves for the integrator example, with $\gamma_0 = 22.91$. 
		Note that $J_{\text{ratio}} = 11.83/4.49 \approx 2.63$, falling inside the bounds $2 \leq J_{\text{ratio}} < 3$ from Theorem \ref{th:bounds}. 
		The result suggests that sampling scheme, rather than controller structure, is the main performance contributor.}
	\label{fig:a0example}
\end{figure}
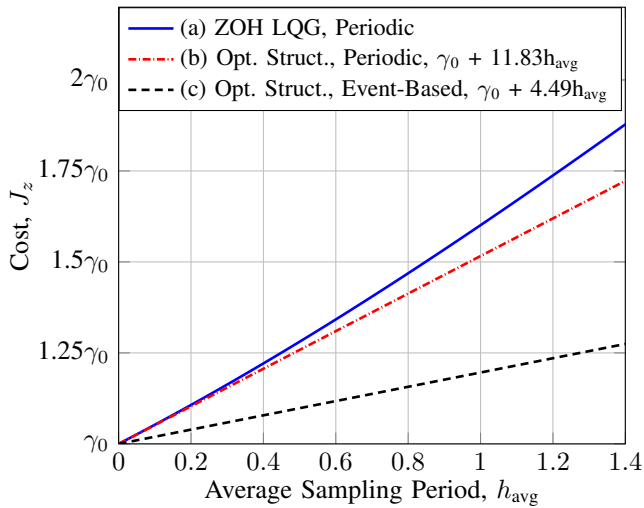

\subsection{Unstable Example}
For the unstable example we use the following parameters:
\begin{align*}
	 A&=\begin{bmatrix}
	 0 & 5 \\
	 5 & 0
	 \end{bmatrix}, 
\quad B_w = C_z^\intercal = \begin{bmatrix}
2.84 & 0 & 0 &0\\
-2.77 & 0.65 & 0 & 0
\end{bmatrix},\\
B_u &= C_y^\intercal = \begin{bmatrix} 
9 & 0\\
8.95 & 0.95
\end{bmatrix}, \quad\quad\,\,\,\, D_{zu}=D_{yw}^\intercal = \begin{bmatrix}
0 \\
I
\end{bmatrix},
\end{align*}
which correspond to $Q=R=I$.
For (c), we numerically obtain optimal trigger bounds for different values of $J$ using the method outlined in Sec.~\ref{sec:numerical_method}. 
The resulting trigger bounds are presented in Fig. \ref{fig:aneq0example_bound}. 
\begin{figure}[tb]
	\centering
	\input{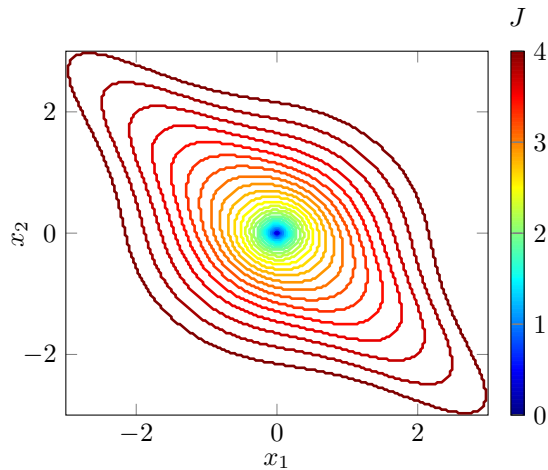}
	\caption{Trigger bounds for the unstable example, using $J \in [5\times10^{-4},\,4]$.}
	\label{fig:aneq0example_bound}
\end{figure}
Trade-off curves for (c) are obtained by simulating a sampled version of the system, with time-step $h_{\text{nom}} = 10^{-3}$, and averaging the cost and sampling period over the time horizon $T=2000$.
The time step $h_{\text{nom}}$ is approximately 130 times smaller than the smallest time constant of the continuous-time LQG closed-loop system.
Note that (c) needs to be simulated in this example, since the average sampling period $h_{\text{avg}}$ cannot be determined directly from $V$ and $J$.
Case (a) is still evaluated using Jitterbug, and (b) is evaluated using \eqref{eq:perpef_vs_samp}. 

The trade-off curves are presented in Fig. \ref{fig:aneq0example}. 
Also this example suggests that sampling scheme rather than controller structure is the main factor for improvement. 
Note that the ratio of improvement $J_{\text{ratio}}$ is increasing with $h_{\text{avg}}$ in this example, whereas  it for the integrator $J_{\text{ratio}}$ was constant. 
This is in line with \cite{kjbob99, meng12}, where similar observations were made for unstable first-order systems and unstable uncoupled second-order systems respectively.
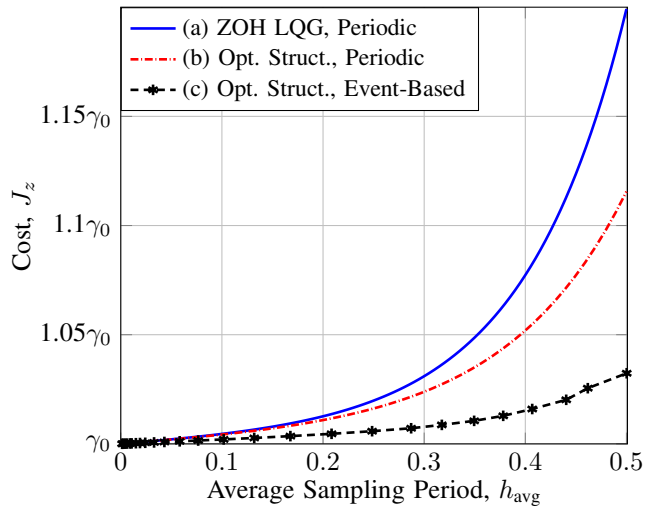
\begin{figure}[tb]
	\centering
%
%
\begin{tikzpicture}

\begin{axis}[%
width=0.78\columnwidth,
scale only axis,
scale only axis,
xmin=0,
xmax=0.501,
xlabel={Average Sampling Period, $h_{\text{avg}}$},
xmajorgrids,
ymin=1,
ymax=1.2,
x label style={at={(axis description cs:0.5,0.03)},anchor=north},
ytick={1,1.05,1.1,1.15},
yticklabels = {$\gamma_0$,1.05$\gamma_0$,1.1$\gamma_0$,1.15$\gamma_0$},
ymajorgrids,
ylabel={Cost, $J_z$},
ylabel style={rotate=0,at={(axis description cs:-0.05,0.5)},anchor=north},
legend style={at={(0.0,1)},anchor=north west,draw=black,fill=white,legend cell align=left,font=\small}
]

\addplot [color=blue,solid,line width=1.0pt]
table[row sep=crcr]{%
1e-06	1.00000003946519\\
0.005	1.00019791182095\\
0.01	1.00039718615519\\
0.015	1.00059809037969\\
0.02	1.00080089790691\\
0.025	1.0010058887258\\
0.03	1.0012133499619\\
0.035	1.00142357645773\\
0.04	1.00163687137481\\
0.045	1.0018535468189\\
0.05	1.0020739244897\\
0.055	1.00229833635679\\
0.06	1.00252712536335\\
0.065	1.00276064615929\\
0.07	1.00299926586569\\
0.075	1.00324336487235\\
0.08	1.00349333767037\\
0.085	1.0037495937219\\
0.09	1.00401255836904\\
0.095	1.0042826737843\\
0.1	1.00456039996483\\
0.105	1.00484621577302\\
0.11	1.00514062002587\\
0.115	1.00544413263604\\
0.12	1.00575729580727\\
0.125	1.00608067528723\\
0.13	1.00641486168092\\
0.135	1.00676047182788\\
0.14	1.00711815024672\\
0.145	1.00748857065052\\
0.15	1.00787243753701\\
0.155	1.0082704878574\\
0.16	1.00868349276819\\
0.165	1.00911225947028\\
0.17	1.00955763314012\\
0.175	1.01002049895766\\
0.18	1.01050178423629\\
0.185	1.0110024606602\\
0.19	1.01152354663469\\
0.195	1.01206610975552\\
0.2	1.0126312694035\\
0.205	1.01322019947081\\
0.21	1.01383413122612\\
0.215	1.0144743563257\\
0.22	1.01514222997809\\
0.225	1.01583917427051\\
0.23	1.01656668166531\\
0.235	1.01732631867547\\
0.24	1.0181197297284\\
0.245	1.01894864122786\\
0.25	1.01981486582439\\
0.255	1.02072030690508\\
0.26	1.02166696331402\\
0.265	1.02265693431557\\
0.27	1.02369242481295\\
0.275	1.02477575083558\\
0.28	1.02590934530889\\
0.285	1.02709576412156\\
0.29	1.02833769250545\\
0.295	1.02963795174454\\
0.3	1.03099950622986\\
0.305	1.03242547087853\\
0.31	1.03391911893565\\
0.315	1.03548389017907\\
0.32	1.03712339954764\\
0.325	1.03884144621529\\
0.33	1.0406420231337\\
0.335	1.04252932706823\\
0.34	1.04450776915227\\
0.345	1.04658198598702\\
0.35	1.04875685131526\\
0.355	1.05103748829856\\
0.36	1.05342928242893\\
0.365	1.05593789510865\\
0.37	1.05856927793176\\
0.375	1.06132968770462\\
0.38	1.06422570224275\\
0.385	1.06726423698562\\
0.39	1.07045256246984\\
0.395	1.07379832270691\\
0.4	1.07730955451122\\
0.405	1.08099470782859\\
0.41	1.08486266711501\\
0.415	1.08892277382393\\
0.42	1.09318485005457\\
0.425	1.09765922342487\\
0.43	1.10235675322931\\
0.435	1.10728885795637\\
0.44	1.11246754421932\\
0.445	1.11790543719539\\
0.45	1.12361581262877\\
0.455	1.12961263050319\\
0.46	1.13591057043244\\
0.465	1.14252506891849\\
0.47	1.14947235850084\\
0.475	1.15676950895866\\
0.48	1.16443447060942\\
0.485	1.17248611991037\\
0.49	1.18094430734683\\
0.495	1.18982990783974\\
0.5	1.19916487375266\\
};
\addlegendentry{(a) ZOH LQG, Periodic};

\addplot [color=red,densely dashdotted,line width=1.0pt]
  table[row sep=crcr]{%
1e-10	1.00000000000395\\
0.0102040817306122	1.00040303991751\\
0.0204081633612245	1.00080817994433\\
0.0306122449918367	1.00121753478833\\
0.040816326622449	1.0016332484489\\
0.0510204082530612	1.00205750915462\\
0.0612244898836735	1.002492564643\\
0.0714285715142857	1.0029407379059\\
0.081632653144898	1.00340444352437\\
0.0918367347755102	1.00388620472145\\
0.102040816406122	1.00438867126757\\
0.112244898036735	1.00491463838027\\
0.122448979667347	1.0054670667682\\
0.132653061297959	1.00604910397915\\
0.142857142928571	1.00666410722252\\
0.153061224559184	1.00731566784928\\
0.163265306189796	1.00800763768614\\
0.173469387820408	1.00874415743648\\
0.18367346945102	1.00952968737785\\
0.193877551081633	1.01036904060516\\
0.204081632712245	1.0112674190904\\
0.214285714342857	1.01223045285332\\
0.224489795973469	1.01326424256393\\
0.234693877604082	1.01437540592678\\
0.244897959234694	1.01557112822911\\
0.255102040865306	1.01685921747027\\
0.265306122495918	1.01824816452902\\
0.275510204126531	1.01974720886802\\
0.285714285757143	1.02136641032248\\
0.295918367387755	1.02311672757154\\
0.306122449018367	1.02501010394869\\
0.31632653064898	1.02705956131019\\
0.326530612279592	1.02927930274984\\
0.336734693910204	1.0316848250245\\
0.346938775540816	1.03429304163875\\
0.357142857171429	1.03712241762867\\
0.367346938802041	1.0401931171863\\
0.377551020432653	1.04352716537723\\
0.387755102063265	1.04714862532623\\
0.397959183693878	1.05108379237995\\
0.40816326532449	1.05536140690351\\
0.418367346955102	1.06001288753008\\
0.428571428585714	1.06507258686083\\
0.438775510216327	1.07057807180891\\
0.448979591846939	1.07657043099643\\
0.459183673477551	1.08309461185049\\
0.469387755108163	1.0901997903047\\
0.479591836738776	1.09793977629904\\
0.489795918369388	1.10637345858538\\
0.5	1.11556529269257\\
};
\addlegendentry{$\text{(b) Opt. Struct., Periodic }$};


\addplot [color=black,line width=1.0pt,densely dashed,mark=asterisk,mark options={solid}]
  table[row sep=crcr]{%
0.00131787032156036	1.00004215813108\\
0.00241086722510388	1.00006224628797\\
0.00420487850003574	1.00009798387239\\
0.00607317569393624	1.00013269762287\\
0.00902995688195589	1.00019165188317\\
0.0107135778528919	1.00022425007566\\
0.0149494707887341	1.00030688372518\\
0.0189967800457822	1.00038692691978\\
0.02419579235171	1.00049073626027\\
0.0329337373205111	1.00066400789855\\
0.0431313349148156	1.0008703132402\\
0.058254689502505	1.00117943962267\\
0.0765345170671973	1.00153470955485\\
0.101750101750102	1.00204368668886\\
0.131952233291548	1.00270651966066\\
0.167714884696017	1.00355859723975\\
0.208441896821261	1.00460843739639\\
0.248508946322068	1.00581071148236\\
0.287108814240597	1.00715879773686\\
0.317460317460317	1.00869483970653\\
0.349406009783368	1.01054401942806\\
0.378143316316884	1.01284306253228\\
0.406834825061025	1.01596710239231\\
0.440431622990531	1.02018461337561\\
0.46178711613946	1.02550518422589\\
0.500250125062531	1.03234859829448\\
};
\addlegendentry{$\text{(c) Opt. Struct., Event-Based}$};

\end{axis}
\end{tikzpicture}%
	\caption{Trade-off curves for the unstable example, with $\gamma_0 = 25.34$. 
		With $\gamma_0$ as zero-reference, the performance ratio between periodic and event-based sampling for the optimal controller structure grows from $J_{\text{ratio}}=1$ at $h_{\text{avg}}=0$ to $J_{\text{ratio}}\approx3.6$ at $h_{\text{avg}}=0.5$.}
	\label{fig:aneq0example}
\end{figure}

\section{CONCLUSION AND FUTURE WORK}
Based on the work on event-based impulse control in \cite{toivo12}, we have presented a framework for finding the optimal event-based sampling scheme for the LQG-optimal sampled-data controller structure. 
Solving the optimization problem for the sampling scheme is equivalent to finding a value function $V$, satisfying a stationary PDE with free boundary.
The function $V$, and thus the optimal sampling scheme, is available in closed form for the multidimensional integrator case.
The sampling scheme then corresponds to an ellipsoidal trigger bound in a linear reset system.
Also, a numerical method has been presented for obtaining $V$ in the second-order general case.
Examples show that the trigger bound is not necessarily convex in the general case.

Future work will be focused on refining the methods of obtaining $V$ for higher-order systems.
As long as no closed-form solution is available, more efficient methods for solving the free boundary problem are required before this is a viable design method for sampled-data control synthesis.
To this end, it is also important to further investigate properties of the optimal trigger bound, such as uniqueness, and whether it is always star-shaped as in all presented examples.  

\section{ACKNOWLEDGMENTS}
The authors would like to thank Toivo Henningsson and Leonid Mirkin for fruitful discussions. Special thanks to Leonid Mirkin for suggesting the simplified scalar search method of Theorem \ref{th:scalarform}.

\newpage
\section*{Appendix}
\label{sec:appendix}
\subsection*{Proof of Theorem \ref{th:scalarform}}
After the state transform $\bar{x}_{\textsc{\tiny H}} = U^\intercal Q^\frac{1}{2}x_{\textsc{\tiny H}}$, \eqref{eq:riccatilike} reads
\begin{align}
	\label{eq:simplify1}
	\bar{P}\bar{R}\bar{P} + \frac{1}{2}\Tr(\bar{R}\bar{P})\bar{P} = I,
\end{align}
where $\bar{P} = U^\intercal Q^{-\frac{\intercal}{2}}PQ^{-\frac{1}{2}}U$. 
With the ansatz $\bar{P} = \text{diag}{\{p_i\}}$, \eqref{eq:simplify1} can be written as
\begin{align}
	\label{eq:simplify2}
	p_i^2r_i + \frac{1}{2}\sum_{j=1}^{n}(p_jr_j)p_i = 1,\quad \forall i\in\{1,\,..,\,n\}.
\end{align}
Note that if there exists a solution to \eqref{eq:simplify2}, then it is the unique solution to \eqref{eq:simplify1}. 
Defining $s := \sum_{j=1}^{n}(p_jr_j) > 0$ and inserting into \eqref{eq:simplify2}, we get after some straightforward calculations
\begin{align}
	\label{eq:root}
    (n+4)s - \sum_{j=1}^{n}\sqrt{s^2 + 16r_j} = 0, \quad s>0.
\end{align}
Solving \eqref{eq:simplify2} is thus equivalent to finding a root of the scalar function $h$, defined as the left hand side of \eqref{eq:root}.
Since $h$ is continuous, with $\lim\limits_{s\rightarrow0^+}h(s) = -\sum_{j=1}^{n}4\sqrt{r_j} < 0$ and $\lim\limits_{s\rightarrow\infty}h(s) > 0$, existence of a solution to \eqref{eq:root} is guaranteed. 
Furthermore, this solution will be unique since $h$ is strictly increasing, with $\frac{dh}{ds} > 4$ for all $s>0$. 
  
\subsection*{Proof of Theorem \ref{th:bounds}}
We have
\begin{align*}
	J_{\text{ratio}} &= 2\frac{\Tr((RP)^2 + \frac{1}{2}\Tr(RP)RP)}{[\Tr(RP)]^2} 
	= 2\frac{\Tr((RP)^2)}{[\Tr(RP)]^2} + 1,
\end{align*}
where \eqref{eq:riccatilike} was used in the first equality. 
Let $\lambda = [\lambda_1,...,\lambda_i,...,\lambda_n]^\intercal$ denote the vector of eigenvalues to $RP$. 
Since $R\succ0$, $P\succ0$, we have $\lambda_i > 0, \forall i$. 
Thus $\Tr(RP) = \sum_{i=1}^n\lambda_i = \sum_{i=1}^n|\lambda_i| = \|\lambda\|_1$.
Also, $\Tr((RP)^2) = \sum_{i=1}^n\lambda_i^2 = \|\lambda\|_2^2$, which gives   
\begin{align*}
J_{\text{ratio}} = 1 + 2\frac{\Tr((RP)^2)}{[\Tr(RP)]^2} = 1 + 2\frac{\|\lambda\|_2^2}{\|\lambda\|_1^2}.
\end{align*}
For the case $n=1$, we have $\|\lambda\|_2^2 = \|\lambda\|_1^2$ and thus $J_{\text{ratio}} = 3$. 
When $n\geq 2$, the Cauchy--Schwarz inequality gives
\begin{align}
\frac{1}{n} \leq \frac{\|\lambda\|_2^2}{\|\lambda\|_1^2} < 1,
\label{eq:vector_ineq}
\end{align}
from which \eqref{eq:bounds} follows. 
The lower bound in \eqref{eq:vector_ineq} is attained when $\lambda_i$ are identical for all $i$, while $J_{\text{ratio}}$ approaches the upper bound when $\lambda_i\rightarrow0$ for all but one $i$. 
To show that the same conditions hold for the eigenvalues of $RQ$, we multiply \eqref{eq:riccatilike} with $R$ from the left
and let $RP = S\mathcal{J}S^{-1}$, where $\mathcal{J}$ is the upper triangular Jordan matrix of $RP$, and $S$ is a similarity transform. 
This gives
\begin{align*}
S(\mathcal{J}^2 + \frac{1}{2}\Tr(\mathcal{J})\mathcal{J})S^{-1} = RQ,
\end{align*}
which means that $RQ$ has the same eigenvalues as $\mathcal{J}^2 + \frac{1}{2}\Tr(\mathcal{J})\mathcal{J}$. 
These eigenvalues are given by $\lambda_i^2 + 1/2\sum_{j=1}^{n}\lambda_j\lambda_i$, and are identical when, and only when, $\lambda_i$ are identical for all $i$.
Since $\lambda_i >0$, the eigenvalues also approach zero when, and only when, $\lambda_i\rightarrow0$.
Thus the conditions on $RQ$ are the same as for $RP$.


\bibliographystyle{IEEEtran}
\bibliography{paper}

\end{document}